\documentclass[10pt]{amsart}
\usepackage{amssymb}
\usepackage{epsfig}
\usepackage{url}
\usepackage{setspace}
\usepackage{mathtools}
\usepackage{enumitem}
\usepackage{xcolor}
\usepackage[justification=centering]{caption}
\usepackage{pdflscape}
\theoremstyle{plain}
\usepackage{tikz}

\usetikzlibrary{automata,positioning,arrows}
\tikzset{->,
>=stealth,
node distance=3cm}

\newtheorem{thm}{Theorem}[section]
\newtheorem{cor}[thm]{Corollary}
\newtheorem{lem}[thm]{Lemma}
\newtheorem{prop}[thm]{Proposition}
\newtheorem{rem}[thm]{Remark}
\newtheorem{ques}[thm]{Question}
\newtheorem{prob}[thm]{Problem}
\newtheorem{conj}[thm]{Conjecture}

\def\cal{\mathcal}
\def\bbb{\mathbb}

\renewcommand{\phi}{\varphi}
\newcommand{\R}{\bbb{R}}
\newcommand{\N}{\bbb{N}}
\newcommand{\Z}{\bbb{Z}}

\makeatletter
\let\@@pmod\pmod
\DeclareRobustCommand{\pmod}{\@ifstar\@pmods\@@pmod}
\def\@pmods#1{\mkern4mu({\operator@font mod}\mkern 6mu#1)}
\makeatother

\begin{document}

\title[Values of binary partition function as sums of three squares]{Values of binary partition function represented by a sum of three squares}
\author{Bartosz Sobolewski and Maciej Ulas}

\subjclass[2010]{}
\thanks{The research of the authors is supported by the grant of the National Science Centre (NCN), Poland, no. UMO-2019/34/E/ST1/00094}

\begin{abstract}
Let $m$ be a positive integer and $b_{m}(n)$ be the number of partitions of $n$ with parts being powers of 2, where each part can take $m$ colors. We show that if $m=2^{k}-1$, then there exists the natural density of integers $n$ such that $b_{m}(n)$ can not be represented as a sum of three squares and it is equal to $1/12$ for $k=1, 2$ and $1/6$ for $k\geq 3$. In particular, for $m=1$ the equation $b_{1}(n)=x^2+y^2+z^2$ has a solution in integers if and only if $n$ is not of the form $2^{2k+2}(8s+2t_{s}+3)+i$ for $i=0, 1$ and $k, s$ are non-negative integers, and where $t_{n}$ is the $n$th term in the Prouhet-Thue-Morse sequence. A similar characterization is obtained for the solutions in $n$ of the equation $b_{2^k-1}(n)=x^2+y^2+z^2$.
\end{abstract}

\maketitle

\section{Introduction}\label{sec1}

Let $\N$ be the set of non-negative integers and $\N_{+}$ the set of positive integers.
Moreover, for a given $n\in\Z$ we define the 2-adic valuation of $n$ as 
$$
\nu_{2}(n)=\max\{k\in\N:\;2^{k}|n\},
$$
with the convention that $\nu_{2}(0)=+\infty$.

The problem of representation of integers by quadratic forms or, more generally, by forms or polynomials in many variables, is a classical one. As was proved by Lagrange in 1770, each non-negative integer can be represented as a sum of four squares. On the other hand, there are infinitely many positive integers which cannot be represented by three squares. More precisely, in 1798 Legendre proved that a non-negative integer $N$ can be represented as
$$
N=x^2+y^2+z^2
$$
for some $x, y, z\in\Z$ if and only if $N$ is not of the form $4^{r}(8s+7)$ for $r, s\in\N$. In particular, the natural density of the set of integers that cannot be represented by a sum of three squares is equal to $1/6$. This raises an interesting question whether for a given sequence of integers $(u_{n})_{n\in\N}$ there exist infinitely many solutions of the Diophantine equation
\begin{equation}\label{uequation}
u_{n}=x^2+y^2+z^2.
\end{equation}
To characterize the solutions of \eqref{uequation} it is necessary to have a good understanding of the 2-adic behavior, or, to be more precise, the 2-adic valuation of the terms of the sequence $(u_{n})_{n\in\N}$. 

Especially interesting is the case where $u_{n}$ has a combinatorial meaning, i.e., $u_{n}$ counts some discrete objects or structures. Equation \eqref{uequation} with $u_{n}=\binom{2n}{n}$ was investigated by Granville and Zhu in \cite{GH}. They characterized those $n\in\N$  such that \eqref{uequation} has no solutions in $x, y, z$. In particular, the set of integers $n$, for which $\binom{2n}{n}$ can be represented as a sum of three squares, has asymptotic density $7/8$ in the set of all natural numbers (for an early study of this problem, see also the paper of Robbins \cite{Rob}). In the same paper they also obtained a characterization of $n \in \N$ such that $n!$ is not a sum of three squares, given in terms of existence of certain patterns in the binary expansion of $n$. A different approach to this problem (via substitutions) was presented by Deshouillers and Luca \cite{DL}. They showed that the natural density of those $n$ such that $n!=x^2+y^2+z^2$ exists and is equal to $7/8$. More precisely, they proved that
$$
\#\{n\leq x:\;n!\;\mbox{is a sum of three squares}\}=\frac{7}{8}x+O(r(x)),
$$
where $r(x)=x^{2/3}$. The error term was improved by Hajdu and Papp \cite{HP} to $r(x)=x^{1/2}\log^{2}x$, and recently by Burns to $r(x)=x^{1/2}$ (see the preprint \cite{Bur}). On the other hand, Robbins obtained a precise characterzation of the solutions in $n$ of \eqref{uequation} in the case where $u_{n}$ is the $n$-th term of Fibonacci or Lucas sequence \cite{Ro}.

In this paper, we follow the same line of research and consider, in particular, equation \eqref{uequation} with $u_{n}$ being the binary partition function $b(n)$. More precisely, $b(n)$ counts the number of partitions of $n \in \N$ into parts being powers of two. For example, $b(4)=4$ because
$$
4=2^2=2+2=1+1+2=1+1+1+1
$$
are all possible representations of 4 as a sum of powers of $2$. The sequence $(b(n))_{n\in\N}$ was already introduced by Euler. However, it seems that the first serious study of its properties was performed by Churchhouse \cite{Chu} in 1969. He computed the 2-adic valuation of $b(n)$. Further results, motivated by Churchhouse's computations, were obtained independently by Gupta \cite{Gup0, Gup1, Gup2} and R{\o}dseth \cite{Rod} (see also recent studies by R{\o}dseth and Sellers \cite{RodSel}). 

Besides being interesting in its own sake, the study of the equation
$$
b(n)=x^2+y^2+z^2
$$
connects two different areas: the Diophantine equations and partitions theory. According to our best knowledge, there is no result in the literature providing a characterization of the solutions of equation \eqref{uequation} with $u_{n}$ being a partition function of sub-exponential growth. Recall that $b(n)$ is indeed of sub-exponential growth. More precisely, Mahler \cite{Mah} proved that $\log_{2} b(n)\sim \frac{1}{2}(\log_{2} n)^{2}$. Our study can be seen as a continuation and extension of recent research concerning solvability of Diophantine equations involving partitions, conducted by Tengely and Ulas \cite{TenUl}. 

In the same context, we also study the $m$-colored binary partition function $b_m(n)$, which counts binary partitions of $n$, where each part can have one of $m \geq 1$ colors. In particular, we have $b(n) = b_1(n)$. The study of arithmetic properties of this function was initiated in a paper of Gawron, Miska and Ulas \cite{GMU}, where several useful results were obtained.

Let us describe the content of the paper in some detail. In Section \ref{sec2} we recall some basic properties and results concerning the function $b_m(n)$.
The main goal of the present investigation, pursued in Sections \ref{sec3}--\ref{sec5}, is to obtain an explicit characterization of the set
$$ S_m = \{n \in \N: b_m(n) \neq x^2+y^2+z^2 \text{ for any } x,y,z \in \Z\}.$$
The reason for considering terms \emph{not} represented as a sum of three squares is that the description turns out to be more concise (similarly as for non-negative integers). We focus primarily on the case $m=2^k-1$ for $k \in \N_+$ due to the fact that the $2$-adic valuation $\nu_2(b_{2^k-1}(n))$ is known to be bounded \cite{Chu, GMU}. This allows us to determine which terms $b_{2^k-1}(n)$ are of the form $4^{r}(8s+7)$ through the reduction modulo a suitable power of $2$.
The analysis is divided into three cases: $k=1, k=2$, and $k \geq 3$, covered in Section \ref{sec3}, \ref{sec4}, and \ref{sec5}, respectively. Based on these results, in Section \ref{sec6} we give quite precise bounds for the counting function of the set $S_{2^k-1}$, and determine its natural density in the process. In the final section, we state some questions, problems, and conjectures which may serve as a basis for further study. In particular, we discuss the equation $b_m(n) = x^2+y^2+z^2$ when $m \neq 2^k-1$ as well as some interesting findings concerning the behavior of $b_m(n)$ modulo powers of $2$. We also collect results of numerical computations related to the equations $b(n)=x^2+y^2+z^4$ and $b(n)=x^2+y^2$.

\begin{rem} \label{rem:other_forms}
{\rm In this paper, we focus on the problem of representation of $b_{2^{k}-1}(n)$ as a sum of three squares. However, it is possible to use our findings to obtain similar results for other ternary quadratic forms $q$ such that the set of $n \in \N$ represented by $q$ is given in terms of the binary expansion of $n$. Such quadratic forms include, for example, the forms $x^2 + y^2 + 2z^2, x^2 + 2y^2 + 2z^2, x^2 + 2y^2 + 4z^2$.} 
\end{rem}




\section{Preliminaries}\label{sec2}

In this section we collect known properties and results which will be used throughout the paper. Recall that the ordinary generating function of the sequence $(b(n))_{n\in\N}$ has the form
$$
B(x)=\prod_{n=0}^{\infty}\frac{1}{1-x^{2^{n}}}=\sum_{n=0}^{\infty}b(n)x^{n}.
$$
As a consequence, we see that $B(x)$ satisfies a Mahler-type functional equation $(1-x)B(x)=B(x^2)$. Comparing the coefficients on both sides, we see that the sequence $(b(n))_{n\in\N}$ satisfies the recurrence: $b(0)=b(1)=1$ and
$$
b(2n)=b(2n-1)+b(n),\quad b(2n+1)=b(2n).
$$
Churchouse \cite{Chu} obtained a characterization of the $2$-adic valuation of the terms $b(n)$.
\begin{thm} \label{thm:chu}
For all $n \geq 2$ we have
$$ \nu_2(b(n)) =
\begin{cases}
1 &\text{if }\nu_{2}(n)\equiv 0\pmod*{2}, \\
2 &\text{if }\nu_{2}(n)\equiv 1\pmod*{2}.
\end{cases}$$
\end{thm}
Another useful result was obtained independently by  R{\o}dseth \cite{Rod} and Gupta \cite{Gup0}, proving a conjecture of Churchhouse. More precisely, we have the following theorem.
\begin{thm} \label{thm:Gupta_Rodseth}
For all $s \in \N$ and odd $n\in\N$ the following congruence holds:
$$
b(2^{s+2}n)\equiv b(2^{s}n)\pmod*{2^{\mu(s)}},
$$
where $\mu(s)=\left\lfloor\frac{3s+4}{2}\right\rfloor$.
\end{thm}

For $m\in\N_{+}$ we define the sequence $(b_{m}(n))_{n \in \N}$ as the convolution of $m$ copies of $(b(n))_{n \in \N}$. More precisely,
$$
b_{m}(n)=\sum_{i_{1}+\ldots+i_{m}=n}b(i_{1})\cdots b(i_{m}).
$$
It is clear that the generating function $B_m(x)$ of the sequence $(b_{m}(n))_{n\in\N}$ is the $m$-th power of $B(x)$, and $b_{m}(n)$ also has a combinatorial interpretation. Indeed, $b_{m}(n)$ is the number of binary partitions of $n$, where each part has one of $m$ possible colors. In a recent paper by Gawron, Miska, and Ulas \cite{GMU}, it is proved that for $m=2^{k}-1$ and $n\geq 2^{k}$ the 2-adic valuation of $b_{m}(n)$ belongs to the set $\{1,2\}$ . More precisely, they gave the following characterization of the $2$-adic valuation of the terms $b_{2^k-1}(n)$.

\begin{thm}[Theorem 4.6 in \cite{GMU}] \label{thm:GMU_valuation}
Let $k \in \N_{+}$. For $n,i \in \N$ such that $i < 2^{k+2}$ we have
$$
\nu_2(b_{2^k-1}(2^{k+2}n+i)) = \begin{cases}
\nu_2(b(8n)) &\text{if } 0 \leq i < 2^k, \\
1            &\text{if } 2^k \leq i < 2^{k+1}, \\
2            &\text{if } 2^{k+1} \leq i < 3 \cdot 2^{k}, \\
1            &\text{if } 3 \cdot 2^{k+1} \leq i < 2^{k+2}.
\end{cases}
$$
In particular, $\nu_2(b_{2^k-1}(n)) \in \{0,1,2\}$ and $\nu_2(b_{2^k-1}(n)) = 0$ if and only if $n < 2^k$.
\end{thm}

The reciprocal of $B(x)$, denoted by
$$
T(x)=\frac{1}{B(x)}=\prod_{n=0}^{\infty}\left(1-x^{2^{n}}\right)=\sum_{n=0}^{\infty}t_{n}x^{n},
$$
is the ordinary generating function for the famous Prouhet-Thue-Morse  sequence $(t_{n})_{n\in\N}$ (PTM sequence for short). Recall that $t_{n}=(-1)^{s_{2}(n)}$, where $s_{2}(n)$ is the number of $1$'s in the unique expansion of $n$ in base 2. Equivalently, we have $t_{0}=1$ and
$$
t_{2n}=t_{n},\quad t_{2n+1}=-t_{n},\quad n\geq 0.
$$
The formula in Theorem \ref{thm:chu} can then be written as $ \nu_{2}(b(n))=\frac{1}{2}|t_{n}-2t_{n-1}+t_{n-2}|$.
We also consider a variant of the PTM sequence, given by
$$
T_{n}=s_{2}(n)\pmod*{2},
$$
i.e., the sequence $(T_{n})_{n\in\N}$ is related to the PTM sequence by $t_{n}=1-2T_{n}$.

The PTM sequence is an example of an {\it automatic sequence}. More precisely, let $k \geq 2$ be a fixed integer. A sequence $\mathbf{a} = (a_n)_{n\in\N}$ is called {\it $k$-automatic} if its {\it $k$-kernel}, namely
$$ K_k(\mathbf{a}) = \{(a_{k^jn+i})_{n\in\N}: j \in \N, 0 \leq i < k^j \},
$$
is a finite set. Equivalently, $\mathbf{a}$ is $k$-automatic if there exists a deterministic finite automaton with output (DFAO) that reads the canonical base-$k$ representation of $n$ and the outputs $a_n$. In the case of the PTM sequence, it is clear that $K_k(\mathbf{t}) =\{\mathbf{t}, -\mathbf{t}\}$. Equivalently, the sequence is generated by the DFAO in Figure \ref{PTM_DFAO}. To generate $t_n$, one moves between the states (symbolized by nodes) according to subsequent digits in the binary representation of $n$. After the final digit has been read, the DFAO returns the output corresponding to the current state. For a detailed treatment of automatic sequences, we refer the reader to the monograph by Allouche and Shallit \cite{AS03a}.
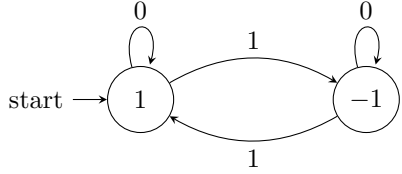
\begin{figure}[h!]
\centering
\begin{tikzpicture}
\node[state, initial](q0){$1$};
\node[state, right of=q0](q1){$-1$};

\draw
 (q0) edge[loop above] node{0} (q0)
 (q0) edge[bend left, above] node{1} (q1)
 (q1) edge[bend left, below] node{1} (q0)
 (q1) edge[loop above] node{0} (q1);
\end{tikzpicture}
\caption{A DFAO generating the PTM sequence.}
\label{PTM_DFAO}
\end{figure}

Finally, let us recall that the sequence ${\bf a}$ with values in $\Z$-module $R$ is $k$-regular if there exist finitely many sequences $\mathbf{a}_{i} = (a_i(n))_{n\in\N}$ with values in $R$ such that each sequence in $K_k({\bf a})$ is a $\Z$-linear combination of the $\bf {a}_{i}$. In other words, the $\Z$-submodule generated by the $k$-kernel $K_k({\bf a})$. In particular, $k$-automatic sequences are precisely $k$-regular sequences taking finitely many values. The class of $k$-regular sequences with values in a ring $R$ has a ring structure itself. A good introduction to the topic of regular sequences are the papers of Allouche and Shallit \cite{AS92, AS03b}.

\section{The equation $b(n)=x^2+y^2+z^2$}\label{sec3}

We start with the characterization of the solutions (in variable $n$) of the equation
$$
b(n)=x^2+y^2+z^2.
$$
Because the values $b(2n)$ and $b(2n+1)$ are equal, we restrict our attention to even indices and consider the set
$$ 
S_1'=\{ n \in \N:  b(2n)\neq x^2+y^2+z^2 \text{ for any } x,y,z \in \Z   \}. 
$$
The first few elements of $S_1'$ are the following:
$$
10, 18, 34, 40, 58, 66, 72, 90, 106, 114, 130, 136, 154, 160, 170, 178, 202, 210, 226, \ldots .
$$

Using Theorem \ref{thm:chu}, we get the following characterization of $\nu_2(b(2n))$.
\begin{prop} \label{prop:u_valuation}
For all $n \in \N_+$ we have
$$\nu_2(b(2n)) = \begin{cases}
1 &\text{if } \nu_2(n) \equiv 0 \pmod{2}, \\
2 &\text{if } \nu_2(n) \equiv 1 \pmod{2}.
\end{cases}$$
\end{prop}
We can deduce that if $\nu_2(n) \equiv 0 \pmod{2}$, then $b(2n)$ is a sum of three squares. Hence, we only need to consider reduction modulo $32$ of  $b(4^k(8m+4))$, where $k,m \in \N$. More precisely, $b(4^k(8m+4))$ is a sum of three squares if and only if
\begin{equation} \label{eq:b_mod32}
    b(4^k(8m+4)) \equiv 28 \pmod{32}. 
\end{equation}
From Theorem \ref{thm:Gupta_Rodseth} and the main result of Hirschhorn and Loxton \cite{HL} one can extract suitable congruence relations, which reduce the general case to $k=0$ and describe the remaining terms $b(8m+4)$.
\begin{prop} \label{prop:b_mod32}
For all $m \in \N$ we have
\begin{align*}
    b(16m) &\equiv b(4m) \pmod{32}, \\
    b(16m+4) &\equiv 4t_m \pmod{32}, \\
    b(16m+12) &\equiv 20t_m \pmod{32}.
\end{align*}
\end{prop}

Using these relations, it is straightforward to describe the set consisting of $n \in \N$ such that $b(2n)$ is (not) a sum of three squares.

\begin{cor}\label{cor:bform}
The following conditions are equivalent:
\begin{enumerate}[label=(\alph*)]
\item The number $b(2n)$ is not a sum of three squares;
\item $n = 2^{2k+1}(4s+1)$ for some $k,s \in \N$ such that $t_s = -1$;
\item $n=2^{2k+1}(8r+2t_{r}+3)$ for some $k,r \in \N$;
\item $\chi(n) = 1$, where $\chi$ is defined by $\chi(0) = 0$ and
$$
\chi(2n+1)=0, \quad \chi(4n)=\chi(n), \quad \chi(8n+2)=T_n, \quad \chi(8n+6)=0.
$$
\end{enumerate}
\end{cor}
\begin{proof}
As we have discussed earlier, $b(2n)$ is not a sum of three squares if and only if $2n = 4^k(8m+4)$ and \eqref{eq:b_mod32} holds. By Proposition \ref{prop:b_mod32} this happens if and only if $m$ is even and $t_m = t_{m/2} = -1$.  Letting $m = 2s$, we obtain the equivalence of (a) and (b).

To prove that (b) is equivalent to (c), we use the following description from \cite{ACS}:
\begin{align} \label{eq:evil_odious}
\{n\in\N:\;T_{n}=0\}&=\{2m+T_m:\;m\in\N\},\\
\{n\in\N:\;T_{n}=1\}&=\{2m+1-T_{m}:\;m\in\N\}. \nonumber
\end{align}
Hence, $t_s=-1$ if and only if $s=2r+1-T_r = 2r+(t_r+1)/2$ for some $r \in \N$, and our claim follows. 


Finally, it is simple to check that the set on $n$ of the form given in (b) is precisely $\{n \in \N: \chi(n) = 1 \}$.
\end{proof}

From the relation $b(2n+1) = b(2n)$ and part (c) of the corollary we get 
$$ S_1 = 2S_1' \cup (2S_1' +1) = \{ 2^{2k+2} (8r+2t_r+3) +i: k,r \in \N, i \in \{0,1\}\}.  $$

Furthermore, part (d) of the corollary directly shows that $S_1'$ (and thus $S_1$) is a $2$-automatic set; i.e., its characteristic sequence $(\chi(n))_{n \in \N}$ is $2$-automatic. A DFAO generating this sequence is shown in Figure \ref{fig:chi_automaton}.

\begin{figure}[h!]
\centering
\begin{tikzpicture}
\node[state, initial](q0){$0$};
\node[state, above right of=q0](q1){$0$};
\node[state, below right of=q0](q2){$0$};
\node[state, below right of=q1](q3){$0$};
\node[state, right of=q3](q4){$0$};
\node[state, right of=q4](q5){$1$};

\draw
 (q0) edge[bend left, above] node{$0$} (q1)
 (q1) edge[bend left, below] node{$0$} (q0)
 (q0) edge[below left] node{$1$} (q2)
 (q2) edge[loop below] node{$0,1$} (q2)
 (q1) edge[above right] node{$1$} (q3)
 (q3) edge[below right] node{$1$} (q2)
 (q3) edge[above] node{$0$} (q4)
 (q4) edge[loop above] node{$0$} (q4)
 (q4) edge[bend left, above] node{$1$} (q5)
 (q5) edge[bend left, below] node{$1$} (q4)
 (q5) edge[loop above] node{$0$} (q5);
\end{tikzpicture}
\caption{A DFAO generating $(\chi(n))_{n \in \N}$}
\label{fig:chi_automaton}
\end{figure}
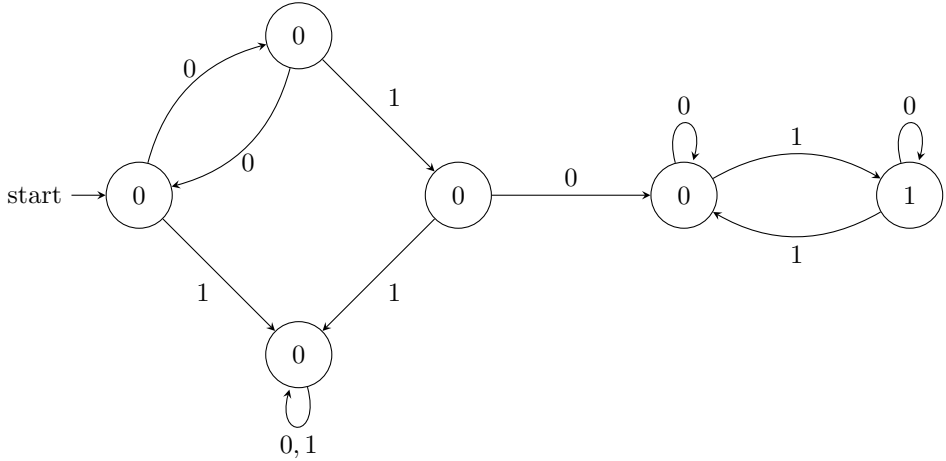

We now turn to the problem of gaps between consecutive $n$ such that $b(2n)$ is a sum of three squares. More precisely, we define $(f_n)_{n \in \N}$ to be the increasing sequence consisting of the elements of $S'$.
Let $(g_n)_{n \in \N}$ be the sequence of gaps, defined by
$$  g_n = f_{n+1} - f_n. $$
In other words, $g_n$ is the distance between $n$th and $(n+1)$th $1$ in the characteristic sequence $(\chi(n))_{n \in \N}$ (counting from $0$).
The following proposition shows that the gaps are bounded. Moreover, for each possible gap length $g$ we provide in the proof an infinite set of $n \in \N$ such that $g_n = g$.
\begin{prop}
For all $n \in \N$ we have
$$ g_n \in \{6,8,10,16,18,24\}  $$
and all possible values are attained infinitely often.
\end{prop}
\begin{proof}
Consider length $16$ subsequences $(\chi(16n+i))_{0 \leq i \leq 15}$. A simple case distinction together with the relations in Corollary \ref{cor:bform}(d) show that there are only four possibilities, namely
$$
\begin{cases}
1, 0, 0, 0, 0, 0, 0, 0, 0, 0, 1, 0, 0, 0, 0, 0 &\text{if } \chi(n) = 1, \\
0, 0, 0, 0, 0, 0, 0, 0, 0, 0, 1, 0, 0, 0, 0, 0 &\text{if } \chi(n) = 0, T_n =1, \\
0, 0, 1, 0, 0, 0, 0, 0, 1, 0, 0, 0, 0, 0, 0, 0 &\text{if } \chi(n) = 0, T_n = 0, 2 \mid n, \\
0, 0, 1, 0, 0, 0, 0, 0, 0, 0, 0, 0, 0, 0, 0, 0 &\text{if } \chi(n) = 0, T_n = 0, 2 \nmid n. \\
\end{cases}
$$
By inspecting the gaps within these subsequences and all their possible concatenations, we can see that the gaps between subsequent $1$'s in $(\chi(n))_{n \in \N}$ can only have lengths $6,8,10,16,18,24$.

It remains to show that each of these values is indeed attained infinitely often. In Table \ref{tab:gaps} for each $g \in \{6,8,10,16,18,24\}$ we provide an infinite set $I_g$ of indices $n$ such that $\chi(n) = \chi(n+g) = 1$ and all the terms inbetween are zero. We note that $I_g$ does not necessarily contain all such indices $n$. The verification of each case is straightforward and left to the reader. 
\end{proof}

\begin{table}[h!]
\begin{tabular}{|l|l|}
\hline
$g$  & $I_g$    \\ \hline 
$6$  & $\{32m + 2: T_m = 1 \}$      \\
$8$  & $\{32m + 10: T_m = 0 \}$ \\
$10$ & $\{16m: \chi(m)  = 1\}$   \\
$16$ & $\{64m + 18: T_m = 0 \}$   \\
$18$ & $\{32m + 8: T_m = 1 \}$     \\
$24$ & $\{256m + 178: T_m = 0 \}$ \\ \hline
\end{tabular}
\caption{Gaps between $1$'s in $(\chi(n))_{n \in \N}$}
\label{tab:gaps}
\end{table}

\begin{rem}
{\rm Let us note that Hajdu and Papp proved that the gap sequence corresponding to those values of $n$ such that $n!$ is a sum of three squares is bounded by 42 \cite[Theorem 2.4]{HP}.
}
\end{rem}

It is also interesting to ask whether the sequence $(f_n)_{n \in \N}$ itself is $2$-regular (equivalently, $(g_n)_{n \in \N}$ is $2$ automatic), since its values form a $2$-automatic set.  This question seems hard, and we have not been able to give a definitive answer (see also Section \ref{sec7}). The problem comes from the fact that the description of elements of $S_1'$ in Corollary \ref{cor:bform} does not give enough information about their ordering. Instead, we consider a simpler version of this question, where we restrict our attention to indices $n \in S_1'$ with fixed $2$-adic valuation. More precisely, we let $k = 0$ in the description of Corollary \ref{cor:bform}, so that $n$ is of the form $4m+2$. Then $b(2n)$ is not a sum of three squares if and only if $b(8m+4) \equiv 28 \pmod{32}$ (this is precisely \eqref{eq:b_mod32} with $k=0$). More generally, put  
$$ \beta(m) = \frac{b(8m+4)}{4} \bmod{8},$$
and for each $a\in\{1,3, 5, 7\}$ let ${\bf c}_{a}=(c_{a}(l))_{l\in\N}$ be the increasing sequence such that
$$
\{m\in\N:\;\beta(m)=a\}=\{c_{a}(l):\;l\in\N\}.
$$
It turns out that these sequences are described by surprisingly simple formulas.

\begin{thm}\label{jvalues}
For each $a\in\{1,3, 5, 7\}$  the sequence ${\bf c}_{a}$ is 2-regular. More precisely, for $m \in \N$ we have
\begin{align*}
c_{1}(l)&=4l-t_{l}+1,\\
c_{3}(l)&=4l+t_{l}+2,\\
c_{5}(l)&=4l-t_{l}+2,\\
c_{7}(l)&=4l+t_{l}+1.
\end{align*}
\end{thm}
\begin{proof}
It is easy to see that each of the sequences from the statement is increasing.

To prove that $\beta(m)=a$ if and only if $m=c_a(l)$ for some $l \in \N$, we restate the second and third relation of Proposition \ref{prop:b_mod32} in the following way:
$$
\beta(m) = \begin{cases}
1 &\text{if } 2 \mid m \text{ and }  t_m = 1, \\
3 &\text{if } 2 \nmid m \text{ and }  t_m = 1, \\
5 &\text{if } 2 \nmid m \text{ and }  t_m = -1, \\
7 &\text{if } 2 \mid m \text{ and }  t_m = -1. \\
\end{cases}
$$
We now use the relations \eqref{eq:evil_odious}. If $\beta(m)=1$, then $2 \mid m$ and $m=2k + T_k$ for some $k \in \N$. This implies $T_k=0$, and thus $k=2l+T_l$ for some $l \in \N$. As a result, we get $m = 4l+2T_l= 4l-t_l+1$. Conversely, if $m$ is of this form, then also $\beta(m)=1$, and so we get the claim for $a=1$.

The proof for $a=3,5,7$ is similar.
\end{proof}

To conclude this section, we point out that similar results can be obtained for other quadratic forms given in Remark \ref{rem:other_forms}. More precisely, depending on the chosen form, they can be derived from either Proposition \ref{prop:b_mod32} or the following set of congruence relations (which again follows from Hirschhorn and Loxton's results).
\begin{prop} \label{prop:b_mod16}
For all $n \in \N$ we have
\begin{align*}
    b(16n+8) &\equiv b(4n+2) \pmod{16}, \\
    b(8n+2) &\equiv 2t_n \pmod{16}, \\
    b(8n+6) &\equiv 6t_n \pmod{16}. 
\end{align*}
\end{prop}

\section{The equation $b_3(n)=x^2+y^2+z^2$}\label{sec4}

In this section we characterize the elements of the set $S_{3}$ containing those $n$ such that $b_{3}(n)$ is not a sum of three squares. By virtue of Theorem \ref{thm:GMU_valuation}, to get the required characterization of $S_{3}$, we need to understand of the behaviour of $b_{3}(16n+i)\mod{32}$ for $i=0, 1, 2, 3, 8, 9, 10, 11$. Let us recall that the sequence $(b_{3}(n))_{n\in\N}$ satisfies the following recurrence relations: $b_{3}(0)=1, b_{3}(1)=3, b_{3}(2)=9$ and
\begin{align*}
b_{3}(2n)&=3b_{3}(2n-1)-3b_{3}(2n-2)+b_{3}(2n-3)+b_{3}(n),\\
b_{3}(2n+1)&=3b_{3}(2n)-3b_{3}(2n-1)+b_{3}(2n-2).
\end{align*}

We start with the following lemma.

\begin{lem}\label{b3congruences}
For all $n \in \N$ the following congruences hold:
\begin{align*}
    b_{3}(8n+i+4)&\equiv 2(2i+1+4(-1)^{n})t_{n}\pmod*{32},\\
    b_{3}(32n+i)&\equiv b_{3}(8n+i)\pmod*{64},\;i=0, 1, 2, 3, 4\\
    b_{3}(8(2n+1)+i)&\equiv 4(3+3i-i^2-2(-1)^{n+i})t_{n}\pmod*{32}\\
                    &\equiv
\begin{cases}
4(3-2(-1)^n)t_{n}\pmod*{32} &\text{if } i=0, \\
4(5+2(-1)^n)t_{n}\pmod*{32} &\text{if } i=1, \\
4(5-2(-1)^n)t_{n}\pmod*{32} &\text{if } i=2, \\
4(3+2(-1)^n)t_{n}\pmod*{32} &\text{if } i=3.
\end{cases}
\end{align*}
In particular, for each $k\in\N_{+}$ and $i\in\{0,1,2,3\}$, we have
\begin{align*}
    b_{3}(2^{2k}(2n+1)+i)&\equiv 2\pmod*{4},\\
    b_{3}(2^{2k+1}(2n+1)+i)&\equiv b_{3}(8(2n+1)+i)\pmod*{32},
\end{align*}
\end{lem}
\begin{proof}
The computation of the values of $b_{3}(8n+i+4)\mod{32} $ and $b_{3}(8(2n+1)+i)\mod{32} $ for $i=0, 1, 2, 3$ is based on a simple induction with the help of recurrence relations satisfied by the PTM sequence $(t_{n})_{n\in\N}$ and the sequence $(b_{3}(n))_{n\in\N}$. Because of this, we omit the simple details. Essentially, the same approach can be used in the case of the congruence $b_{3}(32n+i)\equiv b_{3}(8n+i)\pmod*{32}$.

However, a more conceptual proof is the following. Invoking \cite[Lemma 4.7]{GMU} we know that for each $a\in\N_{+}, b\in\{0, 1, \ldots, 2^{a}-1\}$ there is a polynomial $P_{a,b}\in\Z[x]$ such that
$$
\sum_{n=0}^{\infty}b_{3}(2^{a}n+b)x^{n}=\frac{P_{a,b}(x)}{(1-x)^{3a}}B_{3}(x).
$$
In particular, in the case we are interested in, we have
$$
\sum_{n=0}^{\infty}(b_{3}(32n+i)-b_{3}(8n+i))x^{n}=\frac{P_{5,i}(x)-(1-x)^{6}P_{3,i}(x)}{(1-x)^{15}}B_{3}(x).
$$
A quick computation reveals that for each $i=0, 1, 2, 3$, the polynomial $P_{5,i}(x)-(1-x)^{6}P_{3,i}(x)$ is divisible by 64 in the ring $\Z[x]$. Thus, as the function $(1-x)^{-15}B_{3}(x)$ has power series expansion with integer coefficients, then each number $b_{3}(32n+i)-b_{3}(8n+i)$ is divisible by 64 and we are done.

To obtain the first congruence from the ``in particular'' part, we apply induction on $k$ and the congruence $b_{3}(8n+i+4)\equiv 2(2i+1+4(-1)^{n})t_{n}\pmod*{32}$. In particular, for $i=0, 1, 2, 3$, the 2-adic valuation of $b_{3}(8n+i+4)$ is equal to 1. 

To obtain the congruence $b_{3}(2^{2k+1}(2n+1)+i)\equiv b_{3}(8(2n+1)+i)$ we again use induction on $k$ and apply the congruence $b_{3}(32n+i)\equiv b_{3}(8n+i)\pmod*{64}$.
\end{proof}

We are ready to characterize the set $S_{3}$.

\begin{thm} \label{thm:S3}
We have $n\in S_{3}$ if and only if 
$$
n=2^{2k+1}\left(8p+2\left\lfloor\frac{i}{2}\right\rfloor+3+2(-1)^{i}t_{p}\right)+i
$$
for some $i\in\{0, 1, 2, 3\}$ and $k \in \N_+$, $p\in\N$.
\end{thm}
\begin{proof}
From the characterization of the 2-adic valuation of $b_{3}(n)$ and Lemma \ref{b3congruences} we know that if $n\in S_{3}$, then we necessarily have $n\pmod{16}\in\{0, 1, 2, 3, 8, 9, 10, 11\}$. We perform a case-by-case analysis. 

Let $i\in\{0, 1, 2, 3\}$. If $n\equiv i\pmod*{16}$ and $n=2^{2k}(2s+1)+i$, then $\nu_{2}(b_{3}(n))=1$ and hence $n\not\in S_{3}$. If $n=2^{2k+1}(2s+1)+i$, then we have
$$
b_{3}(n)\equiv b_{3}(8(2s+1)+i)\equiv 4(3+3i-i^2-2(-1)^{s+i})t_{s}\pmod*{32},
$$
and thus $n\in S_{3}$ if and only if $c(i,s):=(3+3i-i^2-2(-1)^{s+i})t_{s}\equiv 7\pmod*{8}$. A case-by-case analysis using the characterization \eqref{eq:evil_odious}, reveals the following:
\begin{enumerate}
\item If $i=0$, then $c(i,s)\equiv 7\pmod*{8}$ if and only if $s$ is even and $t_{s}=-1$. Thus, $s=4p+1+t_{p}$ for some $p\in\N$.
\item If $i=1$, then $c(i,s)\equiv 7\pmod*{8}$ if and only if $s$ is even and $t_{s}=1$. Thus $s=4p+1-t_{p}$ for some $p\in\N$. 
\item If $i=2$, then $c(i,s)\equiv 7\pmod*{8}$ if and only if $s$ is odd, and $t_{s}=1.$ Thus $s=4p+2+t_{p}$ for some $p\in\N$.
\item If $i=3$, then $c(i,s)\equiv 7\pmod*{8}$ if and only if $s$ is odd and $t_{s}=-1.$ Thus $s=4p+2-t_{p}$ for some $p\in\N$.
\end{enumerate}
Gathering all the obtained characterizations, we get the statement of our theorem. 
\end{proof}

\section{The equation $b_{2^k-1}(n)=x^2+y^2+z^2$ with $k \geq 3$} \label{sec5}

In this section, we study for $k \geq 3$ representability of $b_{2^k-1}(n)$ as a sum of three squares. The main idea is to express $(b_{2^k-1}(n))_{n \in \N}$ as the convolution of $(b_{2^k}(n))_{n \in \N}$ and the PTM sequence, and apply the following lemma \cite[Lemma 4.4(1)]{GMU} to $m=2^k$.
\begin{lem} \label{lem:b_m_congruence}
Let $m \in \N_+$. Then for all $n \in \N$ we have
$$ b_{m}(n) \equiv \binom{m}{n} + 2^{\nu_2(m)+1} \binom{m-2}{n-2} \pmod{2^{\nu_2(m)+2}}. $$
\end{lem}

We split our reasoning into two parts: $n < 2^k$ and $n \geq 2^k$. Starting with the simpler case $n < 2^k$, by Theorem \ref{thm:GMU_valuation}, we have $\nu_2(b_{2^k-1}(n)) = 0$. Therefore, it is sufficient for our purposes to describe $b_{2^k-1}(n)$ modulo $8$. 

\begin{prop}\label{prop:b2k_mod8}
Let $k \geq 3$ and $n < 2^k$. Then
$$
b_{2^k-1}(n) \equiv t_n \cdot \begin{cases}
1 \pmod{8} &\text{if } 0 \leq n < 2^{k-2}, \\
5 \pmod{8} &\text{if } 2^{k-2} \leq n < 2^{k-1},\\ 
7 \pmod{8} &\text{if } 2^{k-1} \leq n < 3 \cdot 2^{k-2},\\ 
3 \pmod{8} &\text{if } 3 \cdot 2^{k-2} \leq n < 2^{k}.
\end{cases}
$$
\end{prop}
\begin{proof}
By Lemma \ref{lem:b_m_congruence} we have 
$$ b_{2^k-1}(n) \equiv \sum_{l=0}^n \binom{2^k}{l} t_{n-l} \pmod{8}  $$
Moreover, \cite[Lemma 4.5]{GMU} says that
$$
\binom{2^k}{l} \equiv \begin{cases}
1 \pmod{8} &\text{if } l=0,2^k, \\
4 \pmod{8} &\text{if } l=2^{k-2},3\cdot 2^{k-2}, \\
6 \pmod{8} &\text{if } l=2^{k-1}, \\
0 \pmod{8} &\text{otherwise}.
\end{cases}
$$
From this description we immediately get the claim for the cases $0 \leq n < 2^{k-2}$ and $2^{k-2} \leq n < 2^{k-1}$. If $2^{k-1} \leq n < 3 \cdot 2^{k-2}$, we get
$$ b_{2^k-1}(n) \equiv t_n + 4t_{n-2^{k-2}}+6t_{n-2^{k-1}} \equiv t_n + 2t_{n-2^{k-1}} \pmod{8}. $$
Since $n$ has $2^{k-1}$ in its binary expansion, we get $t_{n-2^{k-1}} = - t_n$, and the required congruence follows. Finally, if $3 \cdot 2^{k-2} \leq n < 2^{k}$, we again have $t_{n-2^{k-1}} = - t_n$ so
$$ b_{2^k-1}(n) \equiv t_n + 4t_{n-2^{k-2}}+6t_{n-2^{k-1}} + 4t_{n-3 \cdot 2^{k-2}} \equiv -5t_n \pmod{8}. \qedhere $$
\end{proof}

As an immediate corollary, we can describe $n < 2^k$ such that $b_{2^k-1}(n)$ is (not) a sum of three squares.

\begin{cor} \label{cor:n<2^k}
Let $k \geq 3$ and $n < 2^k$. Then $b_{2^k-1}(n)$ is not a sum of three squares of integers if and only if one of the following cases holds:
\begin{enumerate}
    \item $0 \leq n < 2^{k-2}$ and $t_n = -1$;
    \item $2^{k-1} \leq n < 3 \cdot 2^{k-2}$ and $t_n = 1$.
\end{enumerate}
\end{cor}

We move on to the case $n \geq 2^k$. This time we have $\nu_2(b_{2^k-1}(n)) \in \{1,2\}$ by Theorem \ref{thm:GMU_valuation}, which means that it is sufficient to consider $b_{2^k-1}(n)$ modulo $32$. To this end, we need a standard lemma concerning the behavior of binomial coefficients modulo powers of $2$ (we provide a proof for completeness).

\begin{lem} \label{lem:binomial_valuation}
The following statements hold:
\begin{enumerate}[label=(\alph*)]
    \item For all $k,n \in \N$ such that $1 \leq n \leq 2^k$, we have 
   $$ \nu_2 \left( \binom{2^k}{n}\right) = k -\nu_2(n). $$
\item For all $m,n \in \N$ we have
 $$ \binom{2m}{2n} \equiv \binom{m}{n} \pmod{2^{\nu_2(m)+1}}.  $$
\end{enumerate}
\end{lem}
\begin{proof}
By Legendre's formula we get
$$ \nu_2 \left( \binom{2^k}{n}\right) = 2^k - 1 - (n - s_2(n) + 2^k-n - s_2(2^k-n)   ) = s_2(n) + s_2(2^k-n) -1. $$
We can express $s_2(2^k-n)$ as 
$$ s_2(2^k-n) = s_2((2^k-1)-(n-1)) = k - s_2(n-1).$$
Now, write $n = 2^{\nu_2(n)} l$, which yields
\begin{align*}
s_2(n-1) &= s_2(2^{\nu_2(n)}(l-1) + (2^{\nu_2(n)}-1)) = s_2(l-1) + \nu_2(n) \\
&= s_2(l) - 1 + \nu_2(n) = s_2(n)-1 + \nu_2(n).
\end{align*}
Combining the above equalities, we get (a).

Moving on to (b), the claim clearly holds for $n=0$ so we can assume that $n \geq 1$. We have the congruence
$$ \binom{2m}{2n} =  \binom{m}{n} \frac{(2m-1)!!}{(2n-1)!!(2m-2n-1)!!} \equiv  (-1)^n \binom{m}{n} \pmod{2^{\nu_2(m)+1}}.$$
If $n$ is even, we immediately obtain (b). If $n$ is odd, we use the inequality
$$ \nu_2\left( \binom{m}{n} \right) = \nu_2 \left( \frac{m}{n} \binom{m-1}{n-1} \right) \geq \nu_2(m),$$
which again leads to the desired result.
\end{proof}

We are now ready to describe $b_{2^k-1}(n)$ modulo $32$ for $n \geq 2^k$. This time, the characterization involves two consecutive terms of the PTM sequence.

\begin{thm} \label{thm:b2k_mod32}
Fix $k,i,j \in \N$ such that $k\geq 3$, $i < 8$, and $j<2^{k-3}$. Then for all $m \geq 1$ we have
$$ b_{2^k-1}(2^km+2^{k-3}i+j) \equiv  t_j(c_i  t_m + d_i t_{m-1}) \pmod{32},    $$
where the coefficients $c_i, d_i$ do not depend on $k$ and are given in Table \ref{tab:cd_i}.
\end{thm}
\begin{table}[h!]
\begin{tabular}{|l|llllllll|}
\hline
$i$ & $0$ & $1$ & $2$ & $3$ & $4$ & $5$ & $6$ & $7$   \\ \hline
$c_i$ & $1$  & $7$  & $3$ & $5$  & $9$  & $-1$  & $3$  & $5$  \\
$d_i$ & $-5$ & $-3$ & $1$ & $-9$ & $-5$ & $-3$  & $-7$ & $-1$ \\ \hline
\end{tabular}
\caption{The coefficients $c_i$, $d_i$.}
\label{tab:cd_i}
\end{table}
\begin{proof}
Consider first the case $k \geq 4$. By Lemma \ref{lem:b_m_congruence} we have
$$
   b_{2^k-1}(n) = \sum_{l=0}^n b_{2^k}(l) t_{n-l} \equiv \sum_{l=0}^n \binom{2^k}{l} t_{n-l} \pmod{32}. 
$$
Now, by Lemma \ref{lem:binomial_valuation}(a), the binomial coefficients with $v_2(l) < k-4$ vanish modulo $32$. Therefore, assuming that $n \geq 2^k$, the above sum simplifies to
$$
b_{2^k-1}(n)  \equiv \sum_{l=0}^{16} \binom{2^k}{2^{k-4}l} t_{n-2^{k-4}l} \equiv \sum_{l=0}^{16} \binom{16}{l} t_{n-2^{k-4}l} \pmod{32},
$$
where the second congruence follows from Lemma \ref{lem:binomial_valuation}(b). Furthermore, we can eliminate the terms with $l$ odd, since there is an even number of them and they are all congruent with $16$ modulo $32$. Therefore, we get the congruence
$$b_{2^k-1}(n) \equiv\sum_{l=0}^{8} \binom{16}{2l} t_{n-2^{k-3}l} \pmod{32}.$$
To simplify the right-hand side, consider $b_{2^k-1}$ at indices of the form given in the statement, namely $n=2^km+2^{k-3}i+j$, where $m \geq 1$, $0 \leq i < 8$, and $0 \leq j < 2^{k-3}$.
For the sake of clarity, we will now momentarily use the notation $t(n) = t_n$. By the recurrences that define the Thue--Morse sequence, we get
$$ t(2^km+2^{k-3}i+j - 2^{k-3}l) = t_j t_{8m+i-l} = t_j \cdot
\begin{cases}
t_n t_{i-l} &\text{if } l \leq i, \\
-t_{n-1}t_{l-i} &\text{if } l > i.
\end{cases} $$
Therefore, the claimed formula is valid with the coefficients
$$
    c_i = \sum_{l=0}^{i} \binom{16}{2l} t_{i-l}, \qquad
    d_i = -\sum_{l=i+1}^{8} \binom{16}{2l} t_{l-i},
$$
and a direct calculation (modulo $32$) gives their values as in Table \ref{tab:cd_i}.

In the case $k=3$, the expression for $b_{2^k-1}(n)$ modulo $32$ obtained from Lemma \ref{lem:b_m_congruence} also contains the sum
$$ 16 \sum_{l=0}^{n} \binom{6}{l-2} t_{n-l}. $$
If $n \geq 8$, then the whole sum vanishes modulo $32$, so we again arrive at the formula
$$
b_{7}(n) \equiv \sum_{l=0}^8 \binom{8}{l} t_{8-l} \pmod{32}. 
$$
After a similar calculation as before, we get the result.
\end{proof}

Using this result, we can determine the indices $n \geq 2^k$ such that $b_{2^k-1}(n)$ is not a sum of three squares. The description turns out to be surprisingly simple.

\begin{cor} \label{cor:n>=2^k}
Let $k \geq 3$ and $n \geq 2^k$. The following conditions are equivalent:
\begin{enumerate}[label=(\alph*)]
    \item $b_{2^k-1}(n)$ is not a sum of three squares;
    \item $t_n = t_{n-2^k} = 1$;
    \item $ n = 2^{k} m + l, $ 
where $l,m \in \N$ are such that $l < 2^k$, $t_m = t_l$, and $\nu_2(m) \equiv 1 \pmod{2}$.
\end{enumerate}
\end{cor}
\begin{proof}
Write $n = 2^km+2^{k-3}i+j$ as in Theorem \ref{thm:b2k_mod32}.
Observe that $c_i + d_i = - 4t_i$, while $c_i - d_i$ is not divisible by $4$. Hence, the term $b_{2^k-1}(2^km+2^{k-3}i+j)$  is not a sum of three squares if and only if
$$ t_m = t_{m-1} = t_i t_j, $$
which after multiplying both sides by $t_i t_j$ gives precisely (b).

The equivalence with (c) is obtained by writing $l = 2^{k-3}i+j$ and observing that $t_m = (-1)^{\nu_2(m)+1} t_{m-1}$.
\end{proof}

\section{Counting the solutions}\label{sec6}
The aim of this section is to provide estimates for the counting functions of the sets $S_{2^k-1}$. For real $x \geq 0$ and $m \in \N_+$ let
$$S_m(x) = S_m \cap [0,x] =  \# \{n \leq x: b_m(n)  \text{ is not a sum of three squares} \}.$$
Using the descriptions of the sets $S_{2^k-1}$ obtained in the previous sections for various $k$ it is straightforward to check that
$$ S_{2^k-1}(x) = \delta_k x + O(\log x), $$
where $\delta_1 = \delta_2 = 1/12$ and $\delta_k = 1/6$ for $k \geq 3$. In the following three results, we provide more precise bounds for $S_{2^k-1}(x) - d_k x$ in the case  $k=1,k=2$ and $k \geq 3$, respectively. In particular, each lower and upper bound is of the form $C_1\log_2 x + C_2$, where the constant $C_1$ is optimal.

\begin{thm}
For every $x \geq 6$ we have
$$
-\frac{5}{3} < S_1(x) - \frac{x}{12} <\frac{1}{2} \log_2 x - \frac{19}{12}.
$$
In particular, the density of the set $S_1$ in $\N$ exists and is equal to $$
\lim_{x\rightarrow +\infty}\frac{S_1(x)}{x}=\frac{1}{12}.
$$
Moreover, there exists an increasing sequence $(m_k)_{k \in \N} \subset \N$ such that $$ S_1(m_l) - \frac{m_l}{12} \sim \frac{1}{2} \log_2 m_l.$$
\end{thm}
\begin{proof}
For real $x \geq 0$ define
\begin{align*}
    P(x) &= \#\{ s \in \N: 8s + 2t_s + 3 \leq x   \}, \\
    Q(x) &= \sum_{k=0}^{\infty} P \left( \frac{x}{4^k} \right).
\end{align*}
By Corollary \ref{cor:bform}(b) and the relation $b(2n+1) = b(2n)$, we get
$$ S_1(x) =  Q \left(\frac{x}{4}\right) + Q \left(\frac{x-1}{4}\right). $$
Hence, it is sufficient to focus on the function $Q$. For $m \in \N$ and $i=0,1,2,3$ 
we have the recurrence relations 
$$ Q(4m+i) = Q(m) + P(4m+i). $$
Also, for $i < 8$ we have
$$
P(8m+i) = m +\begin{cases}
0 &\text{if } i=0, \\
T_m &\text{if } i=1,2,3,4, \\
1  &\text{if } i=5,6,7.
\end{cases}
$$
Put
$$  R(x) = Q(x) - \frac{x}{6},   $$
so that 
$$ S_1(x) -\frac{x}{12} = R\left( \left\lfloor \frac{x}{4} \right\rfloor \right) + R\left( \left\lfloor \frac{x-1}{4} \right\rfloor \right)    + \frac{\left\lfloor \frac{x}{4} \right\rfloor+\left\lfloor \frac{x-1}{4} \right\rfloor}{6} - \frac{x}{12}. $$
It is readily checked that
$$ -\frac{1}{3} < \frac{\left\lfloor \frac{x}{4} \right\rfloor+\left\lfloor \frac{x-1}{4} \right\rfloor}{6} - \frac{x}{12} \leq -\frac{1}{12} $$

Therefore, to obtain the estimates for $ S_1(x) -x/12$ (where $x \geq 9$), it remains to prove that for each integer $m \geq 2$ there holds
$$
    -\frac{2}{3} \leq R(m) \leq \frac{1}{4} \lfloor\log_2 m \rfloor - \frac{1}{4},
$$
as then
$$ -\frac{4}{3} \leq  R\left( \left\lfloor \frac{x}{4} \right\rfloor \right) + R\left( \left\lfloor \frac{x-1}{4} \right\rfloor \right) < 2 \cdot \frac{1}{4} \left( \log_2 \frac{x}{4} - 1 \right) = \frac{1}{2} \log_2 x -\frac{3}{2}. $$
This is done by induction on the length $L(m)$ of the binary expansion of $m$. Direct computation shows that our claim holds when $2 \leq L(m) \leq 6$. Hence, let $L(m) \geq 7$. It is sufficient to prove that there exists an integer $n \geq 2$ with $L(n) \leq L(m) - 2$ such that
$$ 0 \leq R(m) - R(n) \leq \frac{1}{2}.  $$
This is indeed the case, as shown by the following set of identities (ordered according to the residue class modulo $8$):
\begin{align*}
    R(8n) &= R(2n), \\
    R(16n+1)&=R(4n+1), \\
    R(16n+9)&=R(4n) + \frac{1}{2}, \\
    R(16n+2)&=R(4n+2), \\
    R(16n+10)&=R(4n) + \frac{1}{3}, \\
    R(16n+3)&=R(4n+3), \\
    R(16n+11)&=R(4n) + \frac{1}{6}, \\
    R(8n+4)&=R(2n+1) + T_n - \frac{1}{2}, \\
    R(64n+4)&=R(16n+4), \\
    R(64n+20)&=R(16n+2)+1-T_n, \\
    R(64n+36)&=R(16n)+1-T_n, \\
    R(64n+52)&=R(16n+4), \\
    R(16n+12)&=R(4n), \\
    R(8n+5)&=R(2n+1)+\frac{1}{3}, \\
    R(8n+6)&=R(2n+1)+\frac{1}{6}, \\
    R(8n+7)&=R(2n+1).
\end{align*}

We move on to the second part of the statement.
Define $m_0 = 0$ and $m_{l+1} = 16m_l + 36$ for $l \in \N$. First, we prove inductively that $R(m_l)= l$. This is clear for $l = 0$. In general, we have
$$ R(m_{l+1}) = R(16 m_l + 36) = R( 4m_l) + 1 - T_{m_l} = R(m_l) + 1,   $$
where we have used $4 \mid m_l$, the recurrence relations above and $T_{m_l} = 0$ (easily shown by induction). We thus have $S_1(m_0) = S_1(m_1) =0$ and for $l \geq 1$ the equality
\begin{align*}
    S_1(m_{l+1}) - \frac{m_{l+1}}{12} &= R\left( \left\lfloor \frac{m_{l+1}}{4} \right\rfloor \right) + R\left( \left\lfloor \frac{m_{l+1}-1}{4} \right\rfloor \right) - \frac{1}{6} \\
    &= R(m_{l}) + R(m_{l-1}) + 1 = 2l.
\end{align*}
The result follows.
\end{proof}

In Figure \ref{fig:S1} we show the graph of the function $S_1(x) - x/12$ in the range $[1,2^{10}]$ together with the bounds (as in the theorem).
\begin{figure}[h!]
\centering
\includegraphics[width=\textwidth]{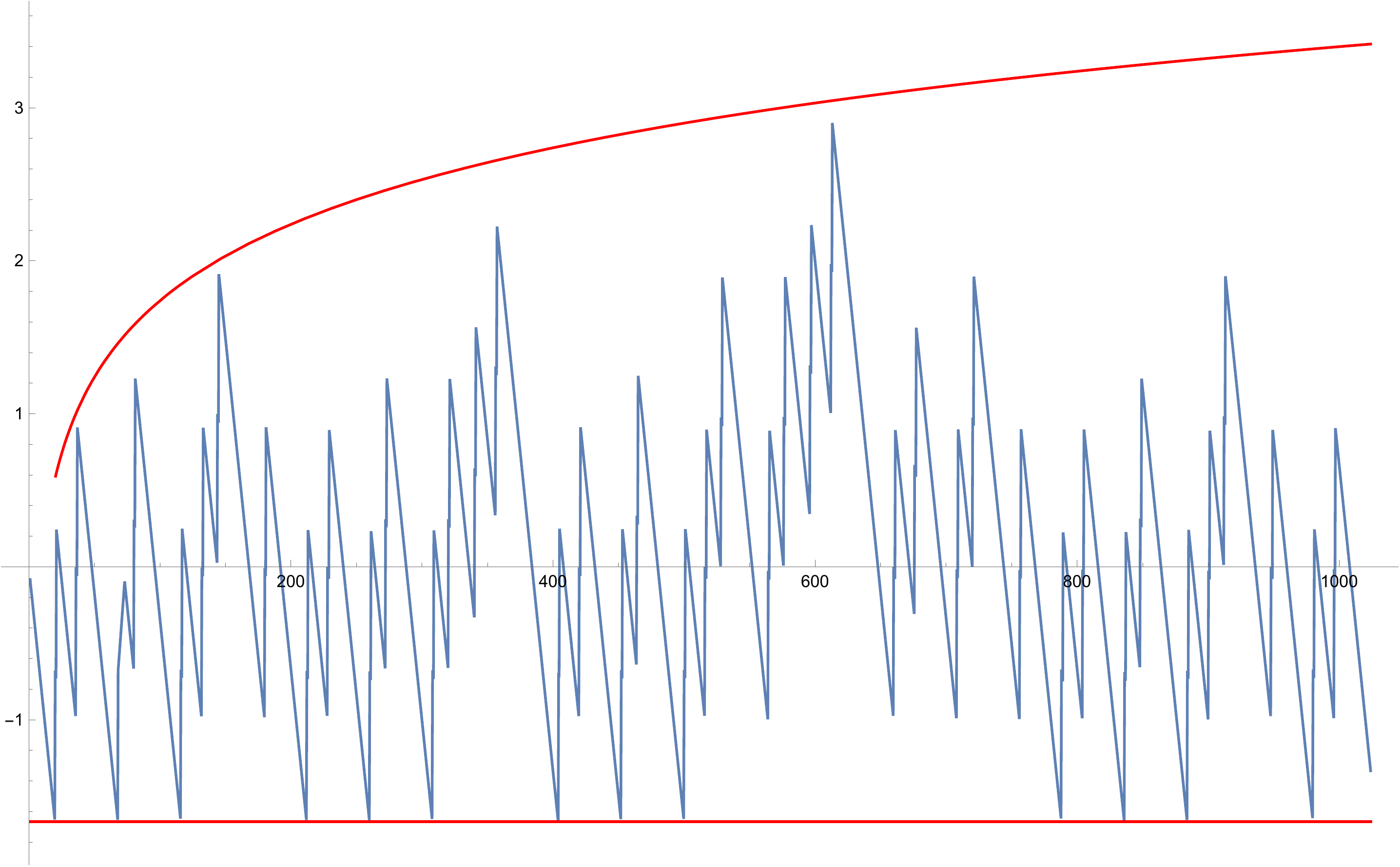}
\caption{The function $S_1(x) - x/12$.}
\label{fig:S1}
\end{figure}
From the presented graph, it appears it should be possible to obtain an even better additive constant in the upper bound. To do this, one would need to investigate closer the location of the ``spikes'' on the graph (some of which correspond to $x = m_l$).

The following two results show that the function $S_1(x)$ is exceptional in the sense that $S_1(x) - x/12$ is bounded from below by a constant. 

\begin{thm} \label{thm:S3(x)}
For all $x \geq 1$ we have
$$-\frac{1}{6} \log_2 x - \frac{7}{12} < S_{3}(x) - \frac{x}{12} \leq \frac{1}{6} \log_2 x - \frac{1}{6}.    $$
In particular, the density of the set $S_3$ in $\N$ exists and is equal to $$
\lim_{x\rightarrow +\infty}\frac{S_3(x)}{x}=\frac{1}{12}.
$$
Moreover, there exist increasing sequences $(m_l)_{l \in \N}, (n_l)_{l \in \N} \subset \N$ such that 
\begin{align*}
    S_3(m_l) - \frac{m_l}{12} &\sim \frac{1}{6} \log_2 m_l, \\
    S_3(n_l) - \frac{n_l}{12} &\sim -\frac{1}{6} \log_2 n_l.
\end{align*}
\end{thm}
\begin{proof}
For $i=0,1,2,3$ let
$$ P_i(x) = \#\{ n \in \N:  8n+2\left\lfloor\frac{i}{2}\right\rfloor+3+2(-1)^{i}t_{n} \leq x\}, $$
so that by Theorem \ref{thm:S3} we have
$$ S_3(x) = \sum_{k=1}^{\infty} \sum_{i=0}^3 P_i \left( \frac{x-i}{2 \cdot 4^k} \right). $$
This time, put
\begin{align*}
  P(x) &= \sum_{i=0}^3 P_i(x), \\
  Q(x) &= \sum_{k=0}^{\infty} P\left(\frac{x}{4^k}\right).
\end{align*}
Then for any $x$ we have
$$ Q\left( \frac{x-3}{8} \right) \leq  S_3(x) \leq Q\left( \frac{x}{8} \right). $$
Therefore, we need to bound the function $R(x) = Q(x) -2x/3$. First, for $n \in \N$ we have the easy to check equalities $P(n) = \lceil n/2 \rceil$  and
$$
R(4n+i) = R(n) + \begin{cases}
0 &\text{for } i =0,3, \\
\frac{1}{3} &\text{for } i =1, \\
-\frac{1}{3} &\text{for } i =2.
\end{cases}
$$
In a similar fashion as in the previous proof, one can then prove that for $m \in \N$ there holds
$$ -\frac{1}{6} \lfloor \log_2 m \rfloor - \frac{1}{6} \leq R(m) \leq  \frac{1}{6} \lfloor\log_2 m \rfloor + \frac{1}{3}  $$
The inequalities for $S_{3}(x)-x/12$ follow shortly by plugging in $m=\lfloor x/8 \rfloor$ and $m=\lfloor (x-3)/8 \rfloor$.

If we define $m_0=0$ and $m_{l+1} = 4m_l + 8$, we can inductively compute $R(m_l/8) = l/3$, and therefore
$$ S(m_l) - \frac{1}{12} m_l \sim  Q\left(\frac{m_l}{8}\right) - \frac{1}{12} m_l = \frac{l}{3} \sim \frac{1}{6} \log_2 m_l.  $$
Similarly, for $n_0 = 0$ and $n_{l+1} = 4n_l + 16$, we get $R(n_l/8) = -l/3$ so
$$ S(n_l) - \frac{1}{12} n_l \sim  -\frac{1}{6}\log_2 n_l,  $$
and the proof is finished.
\end{proof}

\begin{figure}[h!]
\centering
\includegraphics[width=\textwidth]{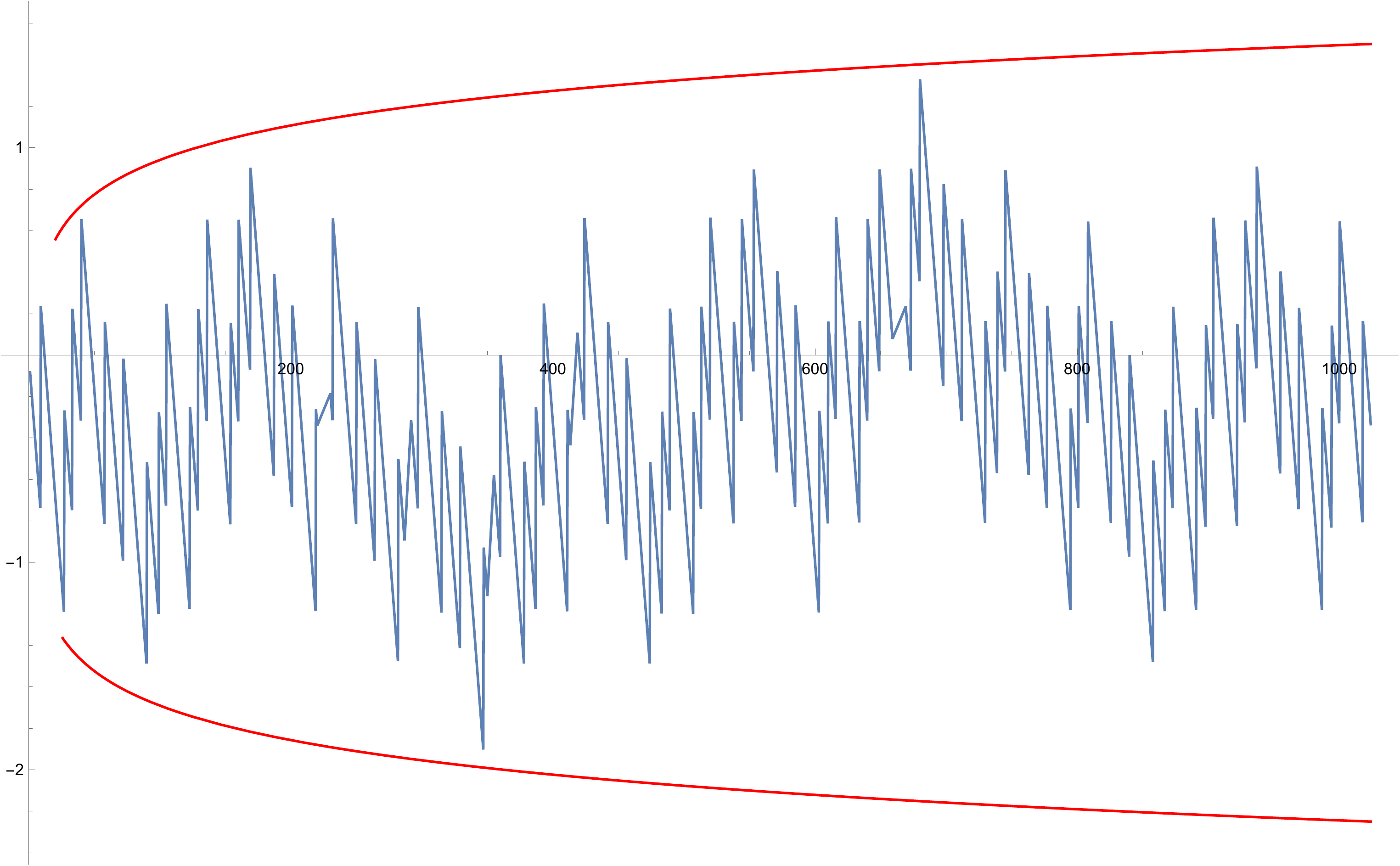}
\caption{The function $S_3(x) - x/12$.}
\label{fig:S3}
\end{figure}
Figure \ref{fig:S3} shows the graph of the function $S_1(x) - x/12$ in the range $[1,2^{10}]$ together with the proved bounds (in red). Again, the bounds are quite accurate, though the additive constants can probably be improved further.

The final result of this section concerns the function $S_{2^k-1}(x)$. For the sake of clarity, in the proof we make some rough estimates concerning the additive constant (although the constant near $\log_2 x$ remains optimal).
\begin{thm} \label{thm:S2k(x)}
If $k \geq 3$, then for all $x \geq 2^k$ we have
$$ \left|S_{2^k-1}(x) - \frac{x}{6} \right| \leq \frac{2^{k-2}}{3} (\log_2 x - k +26).  $$
In particular, the density of the set $S_{2^k-1}$ in $\N$ exists and is equal to $$
\lim_{x\rightarrow +\infty}\frac{S_{2^k-1}(x)}{x}=\frac{1}{6}.
$$
Moreover, there exist increasing sequences $(m_l)_{l \in \N}, (n_l)_{l \in \N} \subset \N$ such that 
\begin{align*}
    S_{2^k-1}(m_l) - \frac{m_l}{6} &\sim \frac{2^{k-2}}{3} \log_2 m_l, \\
    S_{2^k-1}(n_l) - \frac{n_l}{6} &\sim -\frac{2^{k-2}}{3} \log_2 n_l.
\end{align*}
\end{thm}
\begin{proof}
For $\varepsilon \in \{1,-1\}$ and non-negative $x \in \R$ we put
\begin{align*}
    P_{\varepsilon}(x) &= \#\{1 \leq m \leq x: t_m = \varepsilon   \}, \\
    Q_{\varepsilon}(x) &= \sum_{s=0}^{\infty} (-1)^s P_{\varepsilon} \left(  \frac{x}{2^s}  \right) = \#\{1 \leq m \leq x: t_m = \varepsilon \text{ and } \nu_2(m) \equiv 0 \pmod*{2}   \}.
\end{align*} 
Then by Corollary \ref{cor:n>=2^k} we get
\begin{align*}
S_{2^k-1}(x) &= S_{2^k-1}(2^k-1) +  \sum_{l=0}^{2^k-1} \# \{  1 \leq m  \leq \frac{x-l}{2^k}:  t_m = t_l \text{ and } \nu_2(m) \equiv 1 \pmod*{2}  \}  \\
&= 2^{k-2} + \sum_{l=0}^{2^k-1} Q_{t_l}\left(\frac{x-l}{2^{k+1}} \right), 
\end{align*}
where $S_{2^k-1}(2^k-1) = 2^{k-2}$ follows from Corollary \ref{cor:n<2^k}. Furthermore, we have the obvious inequality 
$$ 0 \leq \sum_{l=0}^{2^k-1} Q_{t_l} \left( \frac{x}{2^{k+1}} \right)  - \sum_{l=0}^{2^k-1} Q_{t_l} \left( \frac{x-l}{2^{k+1}} \right) \leq 2^k.   $$
Since for each $\varepsilon = \pm 1$ we have $t_l = \varepsilon$ for precisely $2^{k-1}$ indices $l$, we obtain
\begin{equation}  \label{eq:S2k}
\left| S_{2^k-1}(x)  - 2^{k-1} \left( Q_1\left( \frac{x}{2^{k+1}} \right) +  Q_{-1}\left( \frac{x}{2^{k+1}} \right) \right) \right|
    =5 \cdot 2^{k-2}.
\end{equation}  
Therefore, to bound $ S_{2^k-1}(x) - x/6$ it remains to estimate for $\varepsilon =\pm 1$ the functions
$$ R_{\varepsilon}(x) = Q_{\varepsilon}(x) - \frac{x}{3}.$$
First, note that for $n \in \N$ we have
$$
P_{\varepsilon}(n) = \frac{n-\varepsilon}{2} + \begin{cases} 
\frac{\varepsilon}{2} t_n &\text{if } 2\mid n, \\
0 &\text{if } 2\nmid n.
\end{cases}
$$
It follows that
$$ R_{\varepsilon}(4n+i) = R_{\varepsilon}(n) + \begin{cases}
 0 &\text{if } i=0,3, \\
 \frac{1}{2}(1-\varepsilon t_n) - \frac{1}{3} &\text{if } i=1, \\
  \frac{1}{2}(1-\varepsilon t_n) - \frac{2}{3} &\text{if } i=2.
\end{cases}  $$
This leads to the relations
\begin{align*}
    R_{\varepsilon}(4n) &= R_{\varepsilon}(n), \\
    R_{\varepsilon}(8n+1) &= R_{\varepsilon}(2n+1), \\
    R_{\varepsilon}(16n+5) &= R_{\varepsilon}(n) + \frac{1}{3}, \\
    R_{\varepsilon}(16n+13) &= R_{\varepsilon}(4n+1), \\
    R_{\varepsilon}(16n+2) &= R_{\varepsilon}(4n+2), \\
    R_{\varepsilon}(8n+6) &= R_{\varepsilon}(2n), \\
    R_{\varepsilon}(16n+10) &= R_{\varepsilon}(n)-\frac{1}{3}, \\
    R_{\varepsilon}(4n+3)&=R_{\varepsilon}(n).
\end{align*}
By induction we obtain for $\varepsilon = \pm 1$ and all $n \in \N_+$ the inequality
$$ |R_{\varepsilon}(n)| \leq \frac{1}{12} \lfloor \log_2 n \rfloor + \frac{2}{3},    $$
which implies  
$$ \left| Q_{\varepsilon}(x) - \frac{x}{3} \right| \leq \frac{1}{12}\log_2 x + 1   $$
for all $x \geq 1$. The main part of the result follows shortly.

Finally, put $m_0 = 0$ and $m_{l+1} = 16m_l + 5 \cdot 2^{k+1}$. Also, let $\alpha = (k+1) \bmod{2}$. Using the fact that $2^{k+1} \mid m_l$, from the recurrence relations for $R_{\varepsilon}$ we get
\begin{align*}
    R_{\varepsilon}(2^{\alpha}m_{l+1}) &= R_{\varepsilon}(2^{4+\alpha} m_l + 5 \cdot 2^{k+1 + \alpha})  = R_{\varepsilon}(2^{3-k} m_l + 5) \\
    &= R_{\varepsilon}(2^{-1-k} m_l) + \frac{1}{3} =  R_{\varepsilon}(2^{\alpha}m_l) + \frac{1}{3}.
\end{align*}
It follows that
\begin{align*}  
Q_{\varepsilon}\left( \frac{m_l}{2^{k+1}} \right) - \frac{m_l}{3 \cdot 2^{k+1}} = R_{\varepsilon}(2^{\alpha}m_l) = \frac{l}{3} \sim \frac{1}{12} \log_2 m_l,
\end{align*}
and it remains to use \eqref{eq:S2k}.

Similarly, we can take $n_0 = 0$ and $n_{l+1} = 16n_l + 10 \cdot 2^{k+1}$.
\end{proof}


\section{Computational results, questions, problems and conjectures}\label{sec7}

In this section, we discuss possible directions for further research and present some conjectures and computational results. 

To begin, recall that in Section \ref{sec3} we have defined $(f_n)_{n \in \N}$ to be the increasing sequence such that $S_1' = \{f_n: n \in \N\}$, and asked whether it is regular. We have performed some experimental computations in Mathematica 13 with the help of the \texttt{IntegerSequences} package by Eric Rowland, available at \url{https://ericrowland.github.io/packages.html}. More precisely, for each $m \leq 30$ we have used the  \texttt{FindRegularSequenceRecurrence} function, which did not find a finite set of (plausible) $\Z$-linear relations between the elements of the $m$-kernel $K_m((f_n)_{n \in \N})$. Hence, we expect that following conjecture holds.

\begin{conj}
The sequence $(f_{n})_{n\in\N_{+}}$ is not $m$-regular for any $m \geq 2$. 
\end{conj}

On the other hand, note that we if we consider the decomposition 
$$
S_1'=\bigcup_{k=0}^{\infty}U_{k}
$$
into pairwise disjoint sets $U_{k}=\{2^{2k+1}(8s+2t_{s}+3):\;s\in\N\}$, then for each $k\in\N$ the sequence $(2^{2k+1}(8s+2t_{s}+3))_{s\in\N}$ is 2-regular. 

\bigskip

Next, it is natural to ask whether it is possible to obtain results on the representation of $b_m(n)$ as a sum of three squares for any $m \in \N_+$.
\begin{prob}
Characterize the set $S_m$ for $m \in \N_+$.
\end{prob}
If the valuations $\nu_2(b_m(n))$ are bounded, then the direct approach used in this paper, namely reduction modulo a fixed power of $2$, is sufficient to give a complete description of $S_m$. The following proposition implies that in this case $S_m$ is a $2$-automatic set (its characteristic sequence is $2$-automatic).

\begin{prop} \label{prop:reduction_2p}
For each $m \in \N_+$ and $p \in \N$ the sequence $(b_m(n) \bmod{2^p})_{n \geq 0}$ is $2$-automatic.
\end{prop}
\begin{proof}
Take any  $k \geq p-1$ such that $2^k \geq m$. Note that $(b_m(n))_{n \geq 0}$ is the convolution of the sequence $(b_{2^k}(n))_{n \geq 0}$ with $2^k-m$ copies of the PTM sequence $(t_n)_{n \geq 0}$. They are both $2$-regular when treated as sequences over the ring $\Z/2^p\Z$ (for $(b_{2^k}(n))_{n \geq 0}$ this follows from Lemma \ref{lem:b_m_congruence}). Hence, $(b_m(n) \bmod{2^p})_{n \geq 0}$ is $2$-regular as the convolution of $2$-regular sequences. The result follows from the fact that a $2$-regular sequence attaining finitely many values is necessarily $2$-automatic.
\end{proof}
Unfortunately, we do not know even for a single value $m \neq 2^k-1$, whether or not the valuations $\nu_2(b_m(n))$ are bounded. It is conjectured that they are unbounded for all $m \neq 2^k-1$ (see \cite[Conjecture 5.3]{GMU}). Nevertheless, this does not rule out $2$-automaticity of the set $S_m$. Surprisingly, numerical results for $m \leq 30$ (obtained with help of the \texttt{IntegerSequences} package) suggest that the sets $S_m$ are $2$-automatic for odd $m$, except for $m=17,21$.




It should be possible to get some partial results if we restrict our attention to arithmetic progressions along which $\nu_2(b_m(n))$ is bounded. For example, \cite[Theorem 5.4]{GMU} provides a collection of suitable arithmetic progressions such that $\nu_2(b_2(2^rn +s))$ is constant. By Proposition \ref{prop:reduction_2p}, the set of $n$ such that $b_2(2^rn +s)$ is a sum of three squares, is $2$-automatic.

A related interesting problem concerns the behavior of $b_m(n)$ modulo a fixed power of $2$. 
\begin{prob}
For $m \in \N_+$ and $p \in \N$ characterize $b_m(n) \bmod{2^p}$. 
\end{prob}
We already know that this sequence is $2$-automatic and may ask whether it can be characterized in terms of simpler $2$-automatic sequences. 
The congruences obtained in the previous sections for subsequences of the form $b_{2^k-1}(2^r n + s)$ are all ``admissible'' in the sense of \cite{HL}, that is, only involve $t_n$ and $(-1)^{\nu_2(n)}$. In the case $k \geq 3$ Theorem \ref{thm:b2k_mod32} provides a congruence in terms of $t_n, t_{n-1}$ that can be transformed into an admissible one due to the relation $t_{n-1} = (-1)^{\nu_2(n)+1} t_n$. 

It turns out that other interesting $2$-automatic sequences already appear if we consider $b_m(n)$ modulo suitable powers of $2$. By inspecting modulo $32$ the subsequences described in Proposition \ref{prop:b_mod16},  we have found (without proof) the following set of congruence relations:
\begin{align*}
    b(8n+2) &\equiv 2t_n+16\sigma_n \pmod{32}, \\
    b(8n+6) &\equiv 6t_n + 16\sigma_n+16n \pmod{32}, \\
    b(16n+8) &\equiv (10 + 8n^2)t_n +16\sigma_n \pmod{32},
\end{align*}
where $\sigma_n$ counts modulo $2$ the number of blocks of contiguous $1$'s in the binary expansion of $n$. We have later learned that Alkauskas \cite[Theorem 2]{Alk} obtained a set of relations that describe the same sequences and involve the Rudin--Shapiro sequence instead of $(\sigma_n)_{n \in \N}$. It can be checked that both descriptions are equivalent.

Another sequence that arises in this way is the regular paperfolding sequence $(p_n)_{n \in \N_+}$ defined by $p_{2n} = p_n$ and $p_{2n+1} = (-1)^n$ (see for example \cite[Example 5.1.6]{AS03a}). If we let $P(x)= \sum_{n \geq 1} p_n x^n$, then through manipulation of power series, for $m$ even one can obtain the congruence relation
$$ B_m(x) \equiv (1-x)^m (1+2mP(x)) \pmod*{2^{\nu_2(m)+3}}.  $$
This is essentially a generalization of Lemma \ref{lem:b_m_congruence}.



We now consider some natural modifications of the original equation $b_m(n) = x^2+y^2+y^2$.
We have obtained precise characterization of those $n\in\N$ such that $b(2n)$ is a sum of three squares. In particular, the set of such numbers has natural density equal to $11/12$. Analyzing, for a given $n$ not of the form $2^{2k+1}(8s+2t_{s}+3)$, the solution set $(x, y, z)$ of the equation $b(2n)=x^2+y^2+z^2$, we found that in many cases one of the values $x, y, z$ is a square, i.e., the Diophantine equation
$$
b(2n)=x^2+y^2+z^4
$$
has a solution in non-negative integers. More precisely, for $n\leq 10^3$ we know that there are exactly 916 values of $n$ such that $b(2n)$ is a sum of three squares. Among them, there are exactly 831 values of $n$ such that $b(2n)$ is a sum of two squares and a fourth power. This large number of solutions suggest the following conjecture. 

\begin{conj}
Let $Q_{1}:=\{n\in\N:\;b(2n)=x^2+y^2+z^4\;\mbox{for some}\;x, y, z\in\N\}$. The set $Q_{1}$ is infinite. Moreover, the set $Q_{1}$ has positive natural density in $\N$.
\end{conj}

On the other hand, there are exactly seven values of $n\leq 1000$ such that $b(2n)$ is a sum of a square and two fourth powers. This may suggest that the number of solutions of the equation $b(2n)=x^2+y^4+z^4$ is finite. However, due to limited range of our computations we instead formulate the following:

\begin{ques}
Is the set of $n$ such that $b(2n)=x^2+y^4+z^4$ has a solution in integers $x,y,z$ infinite? 
\end{ques}

\bigskip

An even more interesting and difficult question is whether
the set
$$
\cal{T}=\{n\in\N:\;b(2n)=x^2+y^2\}
$$
is infinite or not. Because we know the behaviour of $b(n)\mod{16}$ we can easily prove that the complement of $\cal{T}$, i.e., $\N\setminus\cal{T}$ is infinite. Indeed, from Proposition \ref{prop:b_mod32} we have $b(16n+4)\equiv 4t_{n}\pmod*{16}$. If $t_{n}=-1$, then $b(16n+4)\equiv 12\pmod*{16}$ and thus $b(16n+4)$ is not a sum of two squares.

To get a clue what can be expected in the case of the set $\cal{T}$, we computed the values of $b(2n)$ for $n\leq 2^{20}$ and check whether $b(2n)$ is a sum of two squares. We put
$$
\cal{T}(x)=\#\{n\leq x:\;n\in \cal{T}\}
$$
and in Table \ref{tab:Tvalues} we present the values of $\cal{T}(2^{n})$ for $n\leq 20$.

\begin{table}[h!]
\begin{tabular}{|c|cccccccccc|}
\hline
$n$ & 1 & 2 & 3 & 4 & 5 & 6 & 7 & 8 & 9 & 10 \\
\hline
$\cal{T}(2^{n})$ & 2 & 3 & 6 & 8 & 14 & 21 & 37 & 64 & 106 & 174 \\
\hline
\hline
$n$ & 11 & 12 & 13 & 14 & 15 & 16 & 17 & 18 & 19 & 20 \\
\hline
$\cal{T}(2^n)$ & 325 & 617 & 1089 & 2018 & 3699 & 6804 & 12551 & 23624 & 44606 & 84176\\
\hline
\end{tabular}
\caption{The number $\cal{T}(2^{n})$ for $n\leq 20$.}
\label{tab:Tvalues}
\end{table}

We also define
$$
\cal{S}=\{r_{2}(b(2n)):\;n\in\N\},
$$
where $r_{2}(m)$ is the number of representations of $m$ as a sum of two squares. Let us recall that
$$
r_{2}(m)=\sum_{d|m,d\equiv 1\pmod*{2}}(-1)^{\frac{d-1}{2}}.
$$
In the considered range; i.e., $n\leq 2^{20}$ the set $\cal{S}$ contains the number 0 and the 35 values $s_{1}\leq \ldots\leq s_{35}$. In Table \ref{tab:slnvalues} below, we present the following values: $s_{i}$, $l_{i}$ -- the number of times $s_{i}$ is attained, and $n_{i}$ -- the smallest value of $n$ such that $r_{2}(b(2n))=s_{i}$.

\begin{table}[h!]
\begin{tabular}{|c|l|l|l||c|l|l|l|}
\hline
 $i$ & $s_{i}$ & $l_{i}$ & $n_{i}$ & $i$ & $s_{i}$ & $l_{i}$ & $n_{i}$\\
 \hline
 1 & 4 & 4 & 0 & 19 & 224 & 1 & 793875 \\
 2 & 8 & 13768 & 4 & 20 & 240 & 1 & 647317 \\
 3 & 12 & 2 & 21 & 21 & 256 & 1005 & 15113 \\
 4 & 16 & 26411 & 30 & 22 & 288 & 13 & 28561 \\
 5 & 24 & 760 & 431 & 23 & 320 & 19 & 113399 \\
 6 & 32 & 22889 & 115 & 24 & 384 & 149 & 24877 \\
 7 & 40 & 36 & 2522 & 25 & 512 & 202 & 11231 \\
 8 & 48 & 1400 & 117 & 26 & 576 & 5 & 420383 \\
 9 & 56 & 1 & 27502 & 27 & 640 & 2 & 210415 \\
 10 & 64 & 11710 & 482 & 28 & 768 & 23 & 88529 \\
 11 & 72 & 9 & 21880 & 29 & 1024 & 27 & 202049 \\
 12 & 80 & 46 & 36642 & 30 & 1152 & 1 & 938983 \\
 13 & 96 & 1094 & 309 & 31 & 1280 & 1 & 162157 \\
 14 & 112 & 2 & 84169 & 32 & 1536 & 5 & 379324 \\
 15 & 128 & 4130 & 1036 & 33 & 2048 & 2 & 324442 \\
 16 & 144 & 9 & 91925 & 34 & 2560 & 1 & 295411 \\
 17 & 160 & 24 & 10785 & 35 & 4096 & 1 & 105400 \\
 18 & 192 & 451 & 3085 &    &      &   &\\
\hline
\end{tabular}
\caption{Values of $s_{i}, l_{i}$ and $n_{i}$ for $i\leq 35$.}
\label{tab:slnvalues}
\end{table}

Our numerical computations suggest the following. 

\begin{conj}
The set $\cal{T}$ is infinite.
\end{conj}

The following heuristic reasoning provides further evidence towards our conjecture. More precisely, recall that the counting function of the sums of two squares up to $x$ is $O(x/\sqrt{\log x})$. Thus, one can say that the probability that a random positive integer $n$ can be written as a sum of two squares of integers is $c/\sqrt{\log n}$. Since, $\log_{2} b(n)\approx \frac{1}{2}(\log_{2}n)^{2}$ one could conjecture that the expectation that $b(n)$ is a sum of two squares is $c'/\log n$ for some positive constant $c'$, provided that $b(n)$ behaves like a random integer of its size. As a consequence, up to $x$, we would have at least 
$$
\sum_{n\leq x} \frac{1}{\log n}=\frac{x}{\log x} +O(x/\log^{2} x)
$$
values of $n$ such that $b(n)$ is a sum of two squares. We dare to formulate the following statement. 

\begin{conj}
There exists a positive real number $c$ such that
$$
\cal{T}(x)=c\frac{x}{\log x}+O(x/\log^{2} x)
$$
as $x\rightarrow +\infty$.
\end{conj}

 Although limited, our computations confirm such an expectation. In Table \ref{tab:TTvalues} we give the values $\cal{T}(2^{m})\frac{m}{2^{m}}$ for $m=10, \ldots, 20$.
\begin{table}[h!]
\begin{tabular}{|c|lllllllllll|}
\hline
 $m$ & 10 & 11 & 12 & 13 & 14 & 15 & 16 & 17 & 18 & 19 & 20 \\
 \hline
 $\cal{T}(2^{m})\frac{m}{2^{m}}$ & 1.67 & 1.74 & 1.80 & 1.73 & 1.72 & 1.7 & 1.66 & 1.63 & 1.62 & 1.62& 1.61 \\
 \hline
\end{tabular}
\caption{Values of $\cal{T}(2^{m})\frac{m}{2^{m}}$ for $m=10, \ldots, 20$.}
\label{tab:TTvalues}
\end{table}

\begin{rem}{\rm The expectation that $b(n)$ behaves like a random integer of its size is very likely. Indeed, numerical computations suggest that for any odd integer $m$ the sequence $(b(n)\pmod{m})_{n\in\N}$ is uniformly distributed; i.e., for any $r\in\{0, 1, \ldots, m-1\}$ we have
$$
\lim_{N\rightarrow +\infty}\frac{\#\{n\leq N:\;b(n)\equiv r\pmod*{m}\}}{N}=\frac{1}{m}.
$$
However, according to the best knowledge of the authors, it is not even known whether the set of prime numbers $p$ such that $p|b(n)$ for some $n$, is infinite.   

}\end{rem}

\bigskip

\bigskip

\noindent Bartosz Sobolewski, Jagiellonian University, Faculty of Mathematics and Computer Science, Institute of Mathematics, {\L}ojasiewicza 6, 30 - 348 Krak\'{o}w, Poland\\
e-mail:\;{\tt  bartosz.sobolewski@uj.edu.pl}
\bigskip

\noindent  Maciej Ulas, Jagiellonian University, Faculty of Mathematics and Computer Science, Institute of Mathematics, {\L}ojasiewicza 6, 30 - 348 Krak\'{o}w, Poland\\
e-mail:\;{\tt  maciej.ulas@uj.edu.pl}
\bigskip

 \end{document}